\documentclass{imsart}
\usepackage{amsfonts}
\usepackage{amsthm}
\usepackage{graphicx} 
\usepackage{epstopdf}
\usepackage[caption=false]{subfig}
\usepackage{cite}

\usepackage{algorithm}
\usepackage{algorithmicx}
\usepackage{algpseudocode}

\usepackage{hyperref}
\ifpdf
\hypersetup{
  colorlinks=true,
  linkcolor=blue,
  citecolor=blue,
  urlcolor=blue,
  pdftitle={},
  pdfauthor={M Roughan} 
}
\fi

\def\equationautorefname~#1\null{%
  (#1)\null
}

\newlength{\figwidth}
\newcommand{\nc}{\newcommand}
\nc{\eg}{{\it e.g., }} 
\newcommand{\ie}{{\em i.e., }}
\nc\E{\ensuremath{\mathbb{E}}}
\nc\V{\ensuremath{\mathbb{V}}}
\nc{\bb}{{\text{\boldmath$b$}}}
\nc{\bs}{{\text{\boldmath$s$}}}
\nc{\eps}{\varepsilon}
\nc{\g}{{$G(N,L)$ }}
\nc{\pp}{($P$)}
\nc{\dd}{($D$)}
\nc{\pps}{($P$)$\,\,$}
\nc{\dds}{($D$)$\,\,$}
\nc{\ep}{{\varepsilon}}
\nc{\bu}{{\text{\boldmath$u$}}}
\nc{\bw}{{\text{\boldmath$w$}}}
\nc{\bq}{{\text{\boldmath$q$}}}
\nc{\be}{{\text{\boldmath$e$}}}
\nc{\bd}{{\text{\boldmath$d$}}}
\nc{\cd}{$c_{ij}^d$}
\nc{\bh}{{\text{\boldmath$h$}}}
\nc{\bv}{{\text{\boldmath$v$}}}
\nc{\bx}{{\text{\boldmath$x$}}}
\nc{\by}{{\text{\boldmath$y$}}}
\nc{\bc}{{\text{\boldmath$c$}}}
\nc{\bo}{{\text{\boldmath$o$}}}
\nc{\bz}{{\text{\boldmath$z$}}}
\nc{\bzero}{{\text{\boldmath$0$}}}
\nc{\R}{\mathbb R}
\nc{\N}{\mathbb N}
\nc{\Z}{\mathbb Z}
\nc{\Q}{\mathbb Q}
\nc{\C}{\mathbb C}
\nc{\mbP}{\mathbb P}
\nc{\I}{\mathbb I}
\nc{\disp}{\displaystyle}
\nc{\br}{{\text{\boldmath$r$}}}
\nc{\bpi}{{\text{\boldmath$\pi$}}}
\nc{\ba}{{\text{\boldmath$a$}}}
\nc{\bp}{{\text{\boldmath$p$}}}
\nc{\ff}{{\mathbf f}}
\nc{\var}{V\!ar}
\nc{\dsum}{\displaystyle \sum}
\nc{\dlim}{\displaystyle \lim}
\nc{\dmin}{\displaystyle \min}
\nc{\dsup}{\displaystyle \sup}
\nc{\dinf}{\displaystyle \inf}
\nc{\cf}{$C(\ff)$}
\nc{\hc}{$\nabla^2C(\ff)$}
\nc{\gc}{$\nabla C(\ff)$}
\nc{\ul}{\underline}
\nc{\zb}{$\bar{\bz}$}
\nc{\pc}{\frac{\partial C(\ff)}{\partial f_e}}
\nc{\xu}{ \bx_{\mu}}
\nc{\dx}{$\Delta \bx$}
\nc{\du}{$\Delta \bx_{\mu}$}

\nc{\p}[1]{\hbox{p} \! \left( #1 \right)}
\nc{\pr}[1]{\hbox{P} \! \left( #1 \right)}
\nc{\prc}[2]{\hbox{P} \! \left( \left. #1 \; \right| #2 \right)}
\nc{\ec}{\end{mathcal}}

\newtheorem{lemma}{Lemma}[section]
\newtheorem{theorem}{Theorem}[section]

\nc\urlb{\begingroup \urlstyle{tt}\Url}

\newcounter{myquote_counter}
\newsavebox{\myquotebox}
\newsavebox{\myquoterefbox}
\newlength{\quotewidth}
\newlength{\quoterefindent}
\newcommand*{\scaledquotefactor}{0.95}
{\renewcommand*{\scaledquotefactor}{1.0}%
   \begin{lrbox}{\myquoterefbox}\begin{minipage}[b]{0.8\quotewidth}%
       \em #1%
   \end{minipage}%
   \end{lrbox}%
   \begin{lrbox}{\myquotebox}\begin{minipage}[b]{1.0\quotewidth} 
}%
{\end{minipage}\end{lrbox}%
\begin{center}%
 \colorbox{quotebackcolor}{\begin{minipage}{\quotewidth}
                        \scalebox{\scaledquotefactor}{\usebox{\myquotebox}}\\[1pt]
          \hspace*{\quoterefindent}\scalebox{\scaledquotefactor}{\usebox{\myquoterefbox}}
       \end{minipage}
      }      
\end{center}%
\par}

\newcommand{\Li}{\mbox{Li}}

\begin{document}

\title{A new non-negative distribution with both finite and infinite support} 
\author{Matthew Roughan\thanks{ARC Centre of Excellence for
    Mathematical \& Statistical Frontiers in the School of
    Mathematical Sciences at the University of Adelaide,
    Australia. \url{matthew.roughan@adelaide.edu.au}} }

\maketitle
 
{\bf Abstract:}  The Tukey-$\lambda$ distribution has interesting properties including (i) for some parameters values it has finite support, and for others infinite support, and (ii) it can mimic several other distributions such that parameter estimation for the Tukey distribution is a method for identifying an appropriate class of distribution to model a set of data. The Tukey-$\lambda$ is, however, symmetric. Here we define a new class of {\em non-negative} distribution with similar properties to the Tukey-$\lambda$ distribution. As with the Tukey-$\lambda$ distribution, our distribution is defined in terms of its quantile function, which in this case is given by the polylogarithm function. We show the support of the distribution to be the Riemann zeta function (when finite), and we provide a closed form for the expectation, provide simple means to calculate the CDF and PDF, and show that it has relationships to the uniform, exponential, inverse beta and extreme-value distributions.





\section{Introduction} 

The Tukey-$\lambda$ distribution
\cite{doi:10.1080/01621459.1971.10482275,10.1145/360827.360840,Karvanen_2008}
is a distribution defined in terms of its quantile function. It has
interesting properties including that (i) for some parameters values
it has finite support, and for others infinite support; and (ii) it
can mimic several other distributions. The latter property makes the
distribution useful for identifying an appropriate class of
distribution to model a set of data. Such properties are rare, as are
definitions of distributions in terms of quantiles. 

In this short note we present a new distribution defined taking its
quantile function to be the polylogarithm function defined by the sum
\begin{equation}
  \label{eq:def}
  Q(p; s) = \Li_s(p) = \sum_{k=1}^{\infty} \frac{p^k}{k^s}.
\end{equation}
Although the polylogarithm is defined (either in terms of this series
or by analytic continuation) for all $p, s \in \C$, we need only the
values $s \in \R$ and $p \in [0,1]$, noting that the defining sum
converges everywhere except for $p = 1$ when $s \leq 1$.

This distributions has similar properties to the Tukey distribution,
save that it defines non-negative random variables where the Tukey
distribution is symmetric. Like the Tukey distribution it interpolates
between several classes of traditional distributions. For 
\begin{itemize}

\item $s$ large ($\geq 10.0$), it closely approximates the uniform distribution,

\item $s \simeq 1.6$, it approximates the (non-negative) triangular
  distribution,

\item $s = 1.0$, it is exactly the exponential distribution,
  
\item $s = 0.0$, it is exactly an inverse beta distribution, and

\item for large negative $s$ it approximates a generalized extreme
  value distribution with infinite mean. 
  
\end{itemize}
This is suggestive that, like the Tukey distribution, this
distribution could be used with a probability plot correlation
coefficient (PPCC) to determine the best distributional family for a
given set of non-negative data. 

This highlights that although it is unusual to specify a class of
random variates in terms of their quantile function -- the inverse
cumulative density function (iCDF) -- there are advantages to doing so.

\section{Background}

\subsection{The Tukey-$\lambda$ distribution}

The Tukey-$\lambda$ distribution
\cite{doi:10.1080/01621459.1971.10482275,10.1145/360827.360840,Karvanen_2008}
(the
  original publication of the distribution was in a technical report
  that no longer appears to be available)  is a distribution defined
by the quantile function 
\[ Q(p; \lambda) = \left\{ \begin{array}{ll}
          \frac{1}{\lambda} \left[ p^\lambda - (1-p)^\lambda \right],
              \mbox{ if } \lambda \neq 0, \\
          \log\left( \frac{p}{1-p} \right), 
              \mbox{ if } \lambda = 0, \\
        \end{array}
              \right.
\]
where $\lambda$ is the shape parameter. The distribution is symmetric
(about 0). 

It is unusual because
\begin{itemize}
\item $\lambda > 0$ the distribution has finite support; and when

\item $\lambda \leq 0$ the distribution has infinite support. 
  
\end{itemize}

The Cumulative Distribution Function (CDF) and Probability Density
Function (PDF) are not given, in general, by closed forms but its
moments are known.

Perhaps most interestingly it interpolates between a range of
traditional distributions: 
\begin{itemize}
\item $\lambda = -1$: it is approximately Cauchy; 
  
\item $\lambda = 0$: it is the logistic distribution; 
  
\item $\lambda = 0.14$: if it approximately normal $N(0, 2.142)$; and
  
\item $\lambda = 1$ and $2$: it is uniform, $U(-1,1)$ and
  $U(-1/2,1/2)$, respectively. 
  
\end{itemize}

A variety of generalization of the distribution exist
\cite{10.1145/360827.360840,Karvanen_2008} and it has been used in a
number of applications (see for example
\cite{doi:10.1080/01621459.1971.10482275,https://doi.org/10.1029/WR015i005p01049}). Its
main purpose here is to serve as basic for comparison to the new distribution
presented in this paper.

\subsection{The polylogarithm function}

The polylogarithm function is a standard special function, albeit not
as commonly used as the gamma or zeta functions. However, it has been
much studied, \eg
\cite{jacobs72:_numer_calcul_polyl,wood92:_comput_poly,crandall06:_note,vepstas08:_hurwit,crandall12:_unified,BAILEY2015115,bailey15:_crandall},
since the 17th century, and there is at least one
book~\cite{lewin81:_polyl} written about its properties. It has many
applications: in this context the most obvious are in the moment
generating function of the zeta distribution and in the mean of the
exponential-logarithmic distribution. There is also a polylogarithmic
distribution~\cite{doi:10.1080/03610919208813032} (not the
distribution specified here).

The polylogarithm is a generalization of the log function in the sense
that the Taylor series of the ordinary logarithm
\[ - \ln(1-z) = \sum_{n=1}^{\infty} \frac{z^n}{n}, 
\]
which matches the definition \autoref{eq:def} for $s=1$, but also note
that this is the quantile function of the exponential distribution.

The polylogarithm function is defined for $s \in \C$; however, we only
consider real parameters here in order to generate real random
variables. 
 
The polylogarithm function is well-behaved in the region of
interest. It is monotonically increasing from 0 at $z=0$ to
\begin{equation}
  \label{eqn:bound}
   \lim_{z \rightarrow 1} \Li_s(z) = \left\{
            \begin{array}{ll}
               \zeta(s), & \mbox{ for } s > 1, \\
               \infty,   & \mbox{ otherwise}, \\
            \end{array}
           \right.
\end{equation}
where $\zeta(s)$ is the well-known Riemann zeta function. Thus the
function forms a valid quantile or inverse CDF function. 

There are a number of computer packages for numerical calculation of
polylogarithms
\eg~\cite{polylog_python,polyl_wolfr_languag_system,vepstas19:_anant}. We
use the package described in \cite{roughan2020polylogarithm}, written
in Julia~\cite{bezanson17:_julia}.

\section{Properties of the Random Variate}

Note that throughout this letter we focus on the single shape
parameter $s$ and neglect separate scale and location parameters, but
these can easily be incorporated by addition and multiplication given
the underlying properties given here. 

\subsection{Support}

From \autoref{eqn:bound} we can immediately identify that the support
of the distribution is
\[
  \mbox{supp}(X) = \left\{
            \begin{array}{ll}
               [0,\zeta(s)], & \mbox{ for } s > 1, \\
               \left[ 0, \infty \right),  & \mbox{ otherwise}. \\
             \end{array}
           \right.
\] 

\subsection{Cumulative distribution and density functions}

We know of no closed form for the inverse of the polylogarithm and so
the CDF must be calculated numerically. However, $Q=F^{-1}$ is
monotonically increasing so this is an easy computation, involving a
one-dimensional search to find the point $p$ such that $F^{-1}(p) = x$
(or equivalently $F(x) = p$).

The density function $f(x)$ can be computed by noting 
\begin{eqnarray*}
   \left. \frac{dF^{-1}}{dp} \right|_{F^{-1}(p) = x} & = & 1 / f(x).
\end{eqnarray*}
As before a search is needed to find the point where $F^{-1}(p) = x$;
however the derivative of the polylogarithm function can be found
directly from the definition
\[ \frac{d}{dx} \Li_s(x) = \frac{ \Li_{s-1}(x) }{x}.
\]
Staudte~\cite{Staudte_2017} argues that this function
  (scaled by the mean where this is finite) should have its own place
  in the important descriptions of a probability distribution and
  calls it the probability density quantile (pdQ) function. He
  proposes and later authors confirm~\cite{dedduwakumara19} that many
  applications can use this function, for instance for parameter
  estimation, more effectively than traditional
  characterizations. From it we obtain the density
\[ f(x) = \left.  \frac{p}{ \Li_{s-1}(p)}\right|_{ \Li_s(p) = x}. 
\]
\autoref{fig:cdf} shows the CDF and density for a range of parameters
$s$. We see several features; notably all of the densities are 1 at the
origin, which fact can be derived from the limit
\[ \frac{ \Li_{s-1}(x) }{x} \rightarrow 1 \mbox{ as } x \rightarrow 0. \]

\begin{figure}
  \begin{center}
    \includegraphics[width=0.48\columnwidth]{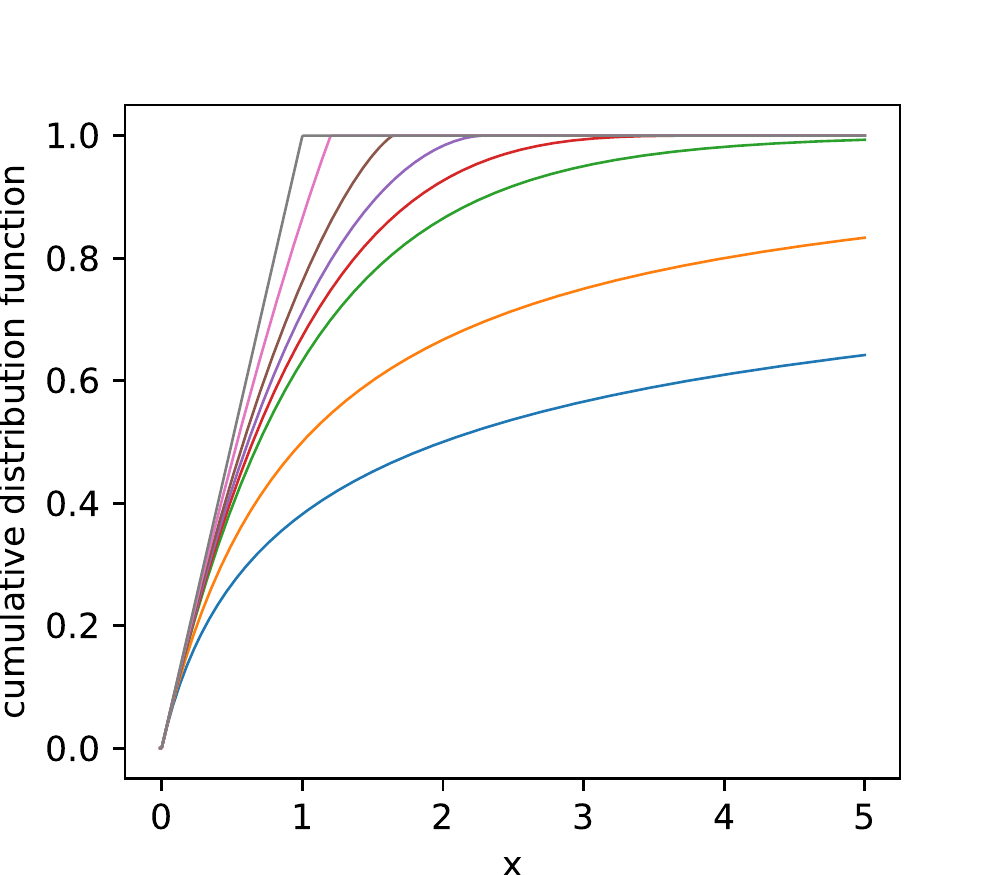}
    \includegraphics[width=0.48\columnwidth]{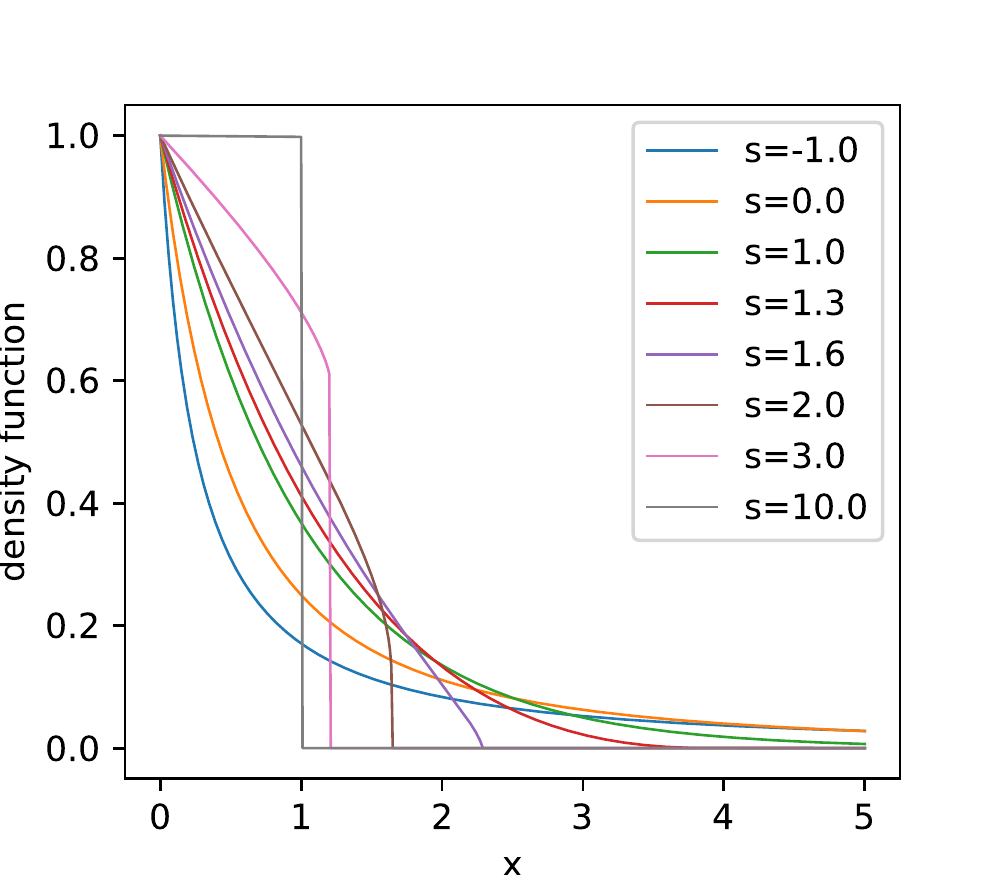}
  \end{center}
  \caption{CDF and PDF}
  \label{fig:cdf}
\end{figure}

\subsection{Simulation}

One advantage of defining a random variable in terms of its quantile
function is that simulation is easy. Given a set of random variables
$X_i \sim U(0,1)$, \ie that are distributed as uniform random
variables on the interval $[0,1]$, then a random variable with quantile
function $Q(p) = F^{-1}(p)$ can be generated by taking
\[ Y_s = Q(X). \]

\subsection{Expectation}

The following result appears almost trivial, but we have not seen it
stated explicitly anywhere.

\begin{lemma}
  The expectation of a random variable $Y$ with quantile function
  $Q(p)$ is given by
  \[ E[ Y ] = \int_0^1 Q(p) \, dp, \]
  when this integral exists.
\end{lemma}
\begin{proof}
  Take random variable $X \sim U(0,1)$, and note that the expectation
  of a $L_1$ function $g(x)$ of a random variable is given
  by \cite[Thm~4.26]{karr92:_probab}
  \[ E[ g(X) ] = \int g(x) dF_X. \]
  Then note that we can create  a random variable $Y$ with quantile function
  $Q(p)$ by taking $Y = Q(X)$ and hence when the integral exists
  \[ E[ Y ] = \int Q(x) dF_X = \int_0^1 Q(p) dp. \]
\end{proof}
The lemma leads directly to the expectation for our random variable $Y
\sim \Li_s(X)$ as being
\[ E[ Y_s ] = \int_0^1 \Li_s(x) dx. \]
Substituting the definition \autoref{eq:def} and taking care that we
integrate over a finite series we get
\begin{eqnarray*}
  E[ Y_s ]
    & = & \lim_{T \rightarrow 1} \sum_{k=1}^{\infty} \frac{1}{k^s} \int_0^T x^k dx \\ 
    & = & \lim_{T \rightarrow 1} \sum_{k=1}^{\infty} \frac{T^{k+1}}{k^s  (k+1) } \\
    & = &  \left\{
            \begin{array}{ll}
               \sum_{k=1}^{\infty}  \frac{1}{k^s  (k+1) }, & \mbox{ for } s > 0, \\
               \infty,   & \mbox{ otherwise}, \\
            \end{array}
           \right.
\end{eqnarray*}
where convergence is derived from $\sum_{k=1}^{\infty} \frac{1}{(k+1)^{s+1}
} \leq \sum_{k=1}^{\infty} \frac{1}{k^s (k+1) } \leq
\sum_{k=1}^{\infty} \frac{1}{k^{s+1} }$.

\begin{figure}
  \begin{center}
    \includegraphics[width=0.7\columnwidth]{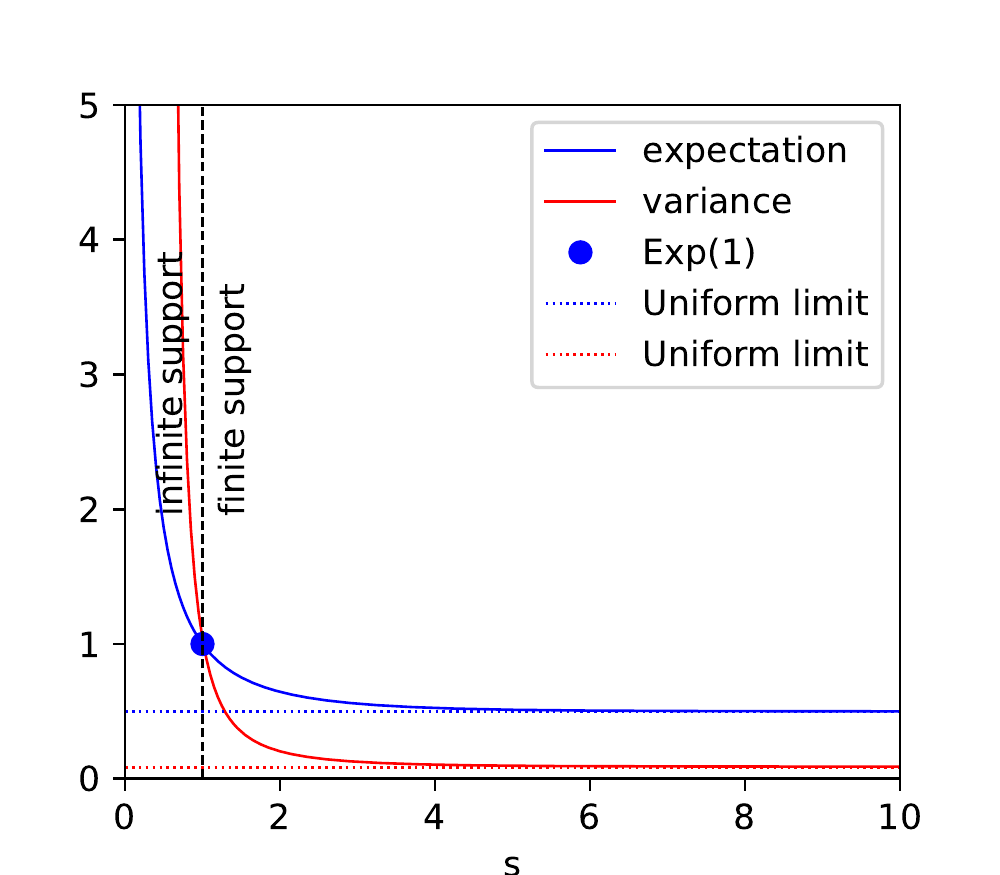}
  \end{center}
  \caption{The solid lines show the expectation and variance. The
    dotted lines show the corresponding limits given by the uniform
    distribution. The dot shows the expectation (and variance) of the
    exponential distribution (with parameter $\lambda=1$). The vertical
    line shows the boundary between regions of finite and infinite
    support. Note that the distribution is also defined for $s < 0$
    but the mean and variance are both infinite in this region.}
  \label{fig:ex}
\end{figure}

\autoref{fig:ex} shows the expectation as a function of $s$.
Interestingly, there is a range of values $0 < s \leq 1$ where the
distribution has infinite support but finite expectation, and a range
where the expectation is infinite $s \leq 0$ where the distribution
resembles the Pareto or power-law distribution. However, the Pareto
distribution does not include parameters with finite support as this
distribution does, nor does it interpolate from the exponential
distribution (at $s=1$).

There are many identities known for polylogarithms. One of use here is
that 
\[ \Li_{s+1}(z) = \int_0^z \frac{\Li_s(t)}{t} dt, \]
which leads via integration by parts to
\[ E[ Y_s ] = \zeta(s+1) - E[ Y_{s+1} ], \]
for random variable $Y_s$ which has quantile function $\Li_s(z)$. This
has advantages for calculating the expectation particularly for $s
\downarrow 0$ where the series converges only very slowly and direct
integration is more problematic because of the pole.

\subsection{Higher moments}

We do not have a closed form for the higher moments, but we can estimate
them through integration, using the above result to note that
\[ E[ Y^m ] = \int_0^1 \Li_s(x)^m dx.
\]
\autoref{fig:ex} also shows the resulting variance computed by
numerical integration. It is noteworthy that the asymptote for
variance appears to be $s_2=1/2$, which we show in the following
result.

\begin{theorem}
  The random variable $Y_s$ with quantile function
  $Q(p;s) = \Li_s(p)$ has infinite $m$th moment for $s < s_m = 1 -
  1/m$. 
\end{theorem}
\begin{proof}
Chebyshev's inequality \cite[Thm~4.40]{karr92:_probab} states that for a
non-negative random variable $X$
\[ P\big(  X \geq a \big) \leq \frac{ E\big[ \varphi( X ) \big] }{\varphi(a)},
\]
for $a \geq 0$ and $\varphi$ a positive, monotonically increasing
function on $\R^+$. Taking $\varphi(x) = x^n$ we get 
\[ P\big(  Y_s \geq a \big) \leq \frac{ E\big[ Y_s^m \big]}{a^m}.
\]
Now take $a = Q_s(1-p)$ and note that $p = P\big(  Y_s \geq
Q_s(1-p)\big)$ to get
\[ Q_s(1-p)^m p  \leq E\big[ Y_s^m \big]. \]
Now the limit of $\Li_s(z)$ near $z=1$ is (for non-integer $s$) given
by \cite{wood92:_comput_poly,lee97:_polyl_rieman,crandall12:_unified} to be
\[ \Li_s(z) \rightarrow \Gamma(1-s) \big( -\ln z \big)^{s-1}.
\]
If we take the first order term in the Taylor series of $\ln z$ we get
\[ \Li_s(1-p) \stackrel{p \rightarrow 0}{\sim} p^{s-1}.
\]
Therefore
\[ Q_s(1-p)^m p  \stackrel{p \rightarrow 0}{\sim} p^{m(s-1) + 1}, \]
and hence for $m=1$ if $s < 0$, the bound pushes the expected value to
$\infty$ as $p$ goes to 0. In general this pushes the $m$th moment to be
infinite when $m(s-1) + 1 < 0$ or for
\[ s < s_m = 1 - 1/m. \]
\end{proof}

\noindent {\bf Remark:} The proof above does not show that the moments are
finite for $s > 1 - 1/m$; however, numerical experiments suggest that
the bound is tight.  The consequence is that we see the asymptote
for the expectation at $s_1 = 0$, the variance at $s_2 = 1/2$, and if
we plot it we see the asymptote for the third moment at $s_3 =
2/3$. The increasing nature of the sequence means that there are
intervals $[1-1/n, 1-1/(n+1))$ where all moments up to $n$ are finite,
  and the $(n+1)$th moment is infinite.

\section{Relationships}

The polylogarithm function has a range of relationships to other
functions most notably the Riemann zeta function \cite[(a)]{lee97:_polyl_rieman}
\[
   \Li_s(1) = \zeta(s), \mbox{ for } \Re(s) > 1.
\] 
There are many other such relationships known, but few are relevant as
they often involve complex parameters or arguments. 

There are several forms of analytic continuation for the polylogarithm
function.  For instance, when $\Re(s)>0$ we can define it using the
integral 
\begin{equation}
  \label{eq:continuation}
  \Li_s(z) =  \frac{1}{z} \int_{0}^{\infty} \frac{t^{s-1}}{e^t/z - 1} dt,
\end{equation}
except for a pole at $z=1$ for $\Re(s) < 2$, which relates the
polylogarithm to the Bose-Einstein distribution, and a similar
relation connects it to the Fermi-Dirac and Maxwell–Boltzmann
distributions. 

More directly useful is that fact that there are a several parameter
values for which the polylogarithm function simplifies dramatically,
\ie  \cite[(3) and (4)]{lee97:_polyl_rieman}
\begin{eqnarray*}
  \label{eq:li_special_case1}
  \Li_1(z)  & = & -\ln(1-z), \\
  \label{eq:li_special_case0}
  \Li_0(z)  & = & z/(1-z), \\
  \label{eq:li_special_case-1}
  \Li_{-1}(z)  & = & z/(1-z)^2,   \\
  \label{eq:li_special_case-2}
  \Li_{-2}(z)  & = & z (1+z)/(1-z)^3,  \\
  \label{eq:li_special_case-3}
  \Li_{-3}(z)  & = & z (1+ +z^2)/(1-z)^4,
\end{eqnarray*}
and so on. In particular the case $s = 1$ leads to the exponential
distribution and $s = 0$ is interesting because $\Li_{0}(z)$ is the
quantile function of the inverse beta distribution (or beta prime
distribution). The PDF of the inverse beta is
\[ iB( \alpha, \beta ) =
    \frac{ x^{\alpha-1} (1 + x)^{-\alpha - \beta}}{ B(\alpha, \beta) }, 
    \]
where $B(\alpha, \beta)$ is the beta function. When $\alpha = 1$ and
$\beta = 1$ we get
\[ iB( 1, 1 ) = \frac{1}{(1 + x)^2},
\]
which has CDF $x / (1+x)$, and inverse CDF $\Li_{0}(z)$. Note that
the mean of the inverse beta is only finite for $\beta > 1$, and that
its mode is at $(\alpha-1)/(\beta + 1) = 0$ as we expect. Also
the inverse beta with these parameters is a power-law type of
heavy-tailed distribution. 


There are a number of useful limits as well \cite{wood92:_comput_poly,lee97:_polyl_rieman,crandall12:_unified}. For instance,
\begin{eqnarray*}
  \lim_{s \rightarrow \infty} \Li_s(z) & = & z,
\end{eqnarray*}
which implies that for large $s$ the distribution tends towards the
uniform distribution. We see this clearly in the above results for
values of $s$ as small as 10.

The limit for negative $s$ is also useful: 
\begin{eqnarray*}
  \lim_{s \rightarrow -\infty} \Li_s(z) & = & \Gamma(1-s) (- \ln(z) )^{s-1}, 
\end{eqnarray*}
which leads to a closed form to calculate the CDF for large negative
$s$. It is also closely linked to the quantile function for the
generalized extreme value distribution for $\xi > 0$
\[ Q_{GEV}(p ; \mu, \sigma, \xi)
    = \mu + \frac{\sigma}{\xi} \left[ \big( -\ln(p) \big)^{-\xi}  - 1 \right], 
\]    
if we identify
\begin{eqnarray*}
  \xi    & = & 1 - s, \\
  \sigma & = & \xi \Gamma(1-s) = \Gamma(2-s), \\
  \mu    & = &  \frac{\sigma}{\xi} = \Gamma(1-s).
\end{eqnarray*}


Another advantage of defining random variates in terms of their
quantile function is that it is then tautologically trivial to
calculate the quantiles. Notably in this case, the median, \ie
$Q(0.5)$, is easy because many values of $\Li_s(0.5)$ are
known, \eg
\begin{eqnarray*}
  \Li_1(0.5) & = & \ln 2, \\
  \Li_2(0.5) & = & \frac{1}{12} \pi^2 - \frac{1}{2} \left( \ln 2 \right)^2.
\end{eqnarray*}
In general $\Li_n(0.5)$ takes values from the multiple zeta function. 

Finally, note that there are a small selection of other random
variates that have the property that their support is both finite or
infinite (apart from the the Tukey lambda distribution mentioned
earlier and its generalizations \cite{10.1145/360827.360840,Karvanen_2008}), \eg
\begin{itemize}
\item the generalized Pareto distribution; and 
  
\item the Wakeby distribution, which can also be defined in terms of
  quantiles, and whose separation into different classes has been used
  to explain phenomena in
  hydrology~\cite{https://doi.org/10.1029/WR015i005p01049}. 
    
\end{itemize}
However, neither of these have the property that they interpolate between
such a range of standard distributions.

\section{Conclusion}

This paper has presented a new random variate that has properties
similar to the Tukey-$\lambda$ distribution except that it is
non-negative.

The distribution is defined in terms of its quantile function, which
is taken to be a polylogarithm function, and from which we can derive
many properties such as the PDF and expectation. It is noteworthy that
this distribution interpolates all the way from the uniform
distribution to heavy-tailed distributions with some set of infinite
moments. 

There are no doubt many generalizations possible --- there are
certainly many generalizations of the polylogarithm function. At a
deeper level this distribution highlights some advantages of defining
a random variate by its quantile.

Finally, the natural name for this distribution might be thought to be
the polylogarithm distribution, however, that name is unfortunately
already taken \cite{doi:10.1080/03610919208813032}, as is the zeta
distribution, and so as yet it remains nameless.

\section*{Acknowledgments}

We would like to thank the Australian Research Council for funding
through the Centre of Excellence for Mathematical \& Statistical
Frontiers (ACEMS), and grant DP110103505. Additional thanks to Giang
Nguyen for helping review this paper.

  
{\small
\bibliographystyle{siamplain}
\bibliography{polylogarithm,books,computation,tukey}
}\par\leavevmode

\end{document}